\documentclass[12pt,reqno]{amsart}

\usepackage{amsaddr} 

\usepackage{a4}

\usepackage{amsfonts}

\usepackage{exscale,txfonts}
\usepackage{mathrsfs}

\usepackage{amssymb,latexsym}
\usepackage{amsmath,amsthm}

\usepackage[all]{xy}

\usepackage[usenames]{color}

\usepackage[normalem]{ulem}

\usepackage{enumerate}

\usepackage{hyperref}

\definecolor{purple}{rgb}{0.49, 0.06, 0.51}

\newcommand{\R}{\mathbb{R}}

\newcommand{\N}{\mathbb{N}}
\newcommand{\Z}{\mathbb{Z}}

\newcommand{\CF}{\mathscr{F}}

\newcommand{\CM}{\mathscr{M}}

\newcommand{\ox}{\otimes}
\newcommand{\x}{\times}

\newcommand{\bs}{\boldsymbol}

\newcommand{\vf}{\varphi}
\newcommand{\ve}{\varepsilon}
\newcommand{\vt}{\vartheta}
\newcommand{\s}{\sigma}

\newcommand{\too}{\longrightarrow}
\newcommand{\mapstoo}{\longmapsto}
\newcommand{\simtoo}{\overset\sim\longrightarrow}

\newcommand{\simm}{\!\!\sim}

\DeclareMathOperator*{\lperp}{\textrm {\Huge$\perp$}}

\newcommand{\df}{\emph}

\renewcommand{\bar}{\overline}

\newcommand{\Xt}{\widetilde{X}}

\newcommand{\id}{\mathrm{id}}

\newcommand{\la}{\langle}
\newcommand{\ra}{\rangle}
\newcommand{\ad}{\mathrm{ad}}

\DeclareMathOperator{\im}{Im}

\DeclareMathOperator{\sign}{sign}
\DeclareMathOperator{\sgn}{sgn}
\DeclareMathOperator{\Sym}{Sym}
\DeclareMathOperator{\rk}{rank}
\DeclareMathOperator{\Hom}{Hom}
\DeclareMathOperator{\Nil}{Nil}
\DeclareMathOperator{\End}{End}

\newtheorem{lemma}{Lemma}[section]
\newtheorem{thm}[lemma]{Theorem}
\newtheorem{prop}[lemma]{Proposition}
\newtheorem{cor}[lemma]{Corollary}

\theoremstyle{definition}
\newtheorem{defi}[lemma]{Definition}

\newtheorem{ex}[lemma]{Example}

\begin{document}
\title{Signatures of hermitian forms and  ``prime ideals'' of Witt groups}
\author{Vincent Astier and Thomas Unger}

\address{School of Mathematical Sciences\\ University College Dublin\\ Belfield,
Dublin~4, Ireland} 
\email{vincent.astier@ucd.ie, thomas.unger@ucd.ie}


\begin{abstract}
In this paper a further study is made of $H$-signatures of hermitian forms, introduced 
previously by the authors. 
It is shown that a tuple of reference forms $H$ may be replaced by a single form and
that the $H$-signature is  invariant under Morita equivalence of algebras with involution. The 
``prime ideals'' of the Witt group are studied, obtaining results that are analogues of the classification of prime ideals of the Witt ring  by Harrison and Lorenz-Leicht.  
It follows that $H$-signatures canonically correspond to morphisms into the integers.
\smallskip

\end{abstract}

\maketitle

\section{Introduction}

Let $F$ be a formally real field and let $(A,\s)$ be an $F$-algebra with involution. In \cite{AU} the notion of $H$-signature 
$\sign_P^H h$
of a hermitian form $h$ over $(A,\s)$ with respect to an ordering $P$ on $F$ is defined. This is a refinement of the definition of
signature $\sign_P h$ in \cite{BP}. Both signatures are defined via scalar extension to a real closure $F_P$ of $F$ at $P$, followed 
by an application of Morita theory, which reduces the computation to what essentially is the case of classical signatures of quadratic forms. 

Pfister's local-global principle holds for both signatures since $\sign_Ph=0$ if and only if $\sign_P^H 
h=0$, see \cite{LU}, and the Knebusch trace formula holds for $H$-signatures, see \cite{AU}.

Both $\sign_P^H$ and $\sign_P$ should be considered as relative invariants of forms in the sense that knowledge about the non-triviality of the signature of 
some fixed reference form(s) is needed in order to compute the signature of an arbitrary form. In \cite{BP} this reference form is taken to be
the unit form $\la 1\ra_\s$. As  was demonstrated in \cite[3.11]{AU}, $\sign_P\la 1\ra_\s=0$ whenever $\s$ becomes hyperbolic over the  extended algebra $A\ox_F F_P$ and thus $\la 1\ra_\s$ cannot be used as a reference form in this case. In \cite[6.4]{AU} the existence of a finite tuple $H$ of reference  forms was established that does not suffer from this problem.

After recalling some definitions and properties, we show
in this paper that the tuple $H$ can be replaced by just one reference form that has non-zero $H$-signature with respect to all relevant orderings on $F$. In the remainder of the paper we show that $H$-signatures are invariant under Morita equivalence of $F$-algebras with involution and we 
study the ``prime ideals'' of the Witt group $W(A,\s)$, obtaining results that are analogues of the classification of prime ideals of the Witt ring $W(F)$ by Harrison \cite{H} and Lorenz and Leicht \cite{LL}.  
As a consequence we show 
that $H$-signatures are a canonical notion in the sense that they
can be identified with a  natural  set of morphisms from $W(A,\s)$ to $\Z$.

\section{Notation}

We give a brief overview of the notation we use and refer to the standard references \cite{Knus}, \cite{BOI}, \cite{Lam} and \cite{Sch} as well as to \cite{AU} for the details.

Let $F$ be a field of characteristic different from $2$. We write $W(F)$ for the \df{Witt ring} of Witt equivalence classes of quadratic forms over 
$F$ and $I(F)$ for its \df{fundamental ideal}. 

By an \df{$F$-algebra with involution} we mean a pair $(A,\s)$ where $A$ is a finite-di\-men\-sional $F$-algebra, whose centre $Z(A)$ satisfies $[Z(A):F]=:\kappa\leq 2$, and 
which is assumed to be either simple (if $Z(A)$ is a field) or a direct product 
of two simple algebras (if $Z(A)=F\x F$) and $\s:A\to A$ is an $F$-linear involution. Let $\Sym(A,\s)=\{a\in A\mid \s(a)=a\}$ and $\Sym(A,\s)^\x=\Sym(A,\s)\cap A^\x$
denote the set of symmetric elements and the set of invertible symmetric elements in $A$ with respect to $\s$, respectively. When $\kappa=1$, we say that 
$\s$ is \df{of the first kind}. When $\kappa=2$, we say that $\s$ is \df{of the second kind} (or \df{unitary}).  Assume $\kappa=1$ and $\dim_FA=m^2\in \N$. Then $\s$ is \df{orthogonal} or \df{symplectic} if $\dim_F \Sym(A,\s)= m(m+1)/2$ or $m(m-1)/2$, respectively. 

For $\ve\in\{-1,1\}$ we write $W_\ve(A,\s)$ for the \df{Witt group} of Witt equivalence classes of $\ve$-hermitian forms $h:M\x M\to A$, defined on 
finitely generated right $A$-modules $M$. All forms in this paper are assumed to be non-singular and are identified with their Witt equivalence classes, unless indicated otherwise.
We write $+$ for the sum in both $W(F)$ and $W_\ve(A,\s)$. The group 
$W_\ve(A,\s)$ is a $W(F)$-module and we denote the product of $q\in W(F)$ and $h\in W_\ve(A,\s)$ by $q\cdot h$.

From the structure theory of $F$-algebras with involution it follows that $A$ is isomorphic to a full matrix algebra $M_n(D)$ for a unique $n\in \N$ and an $F$-division algebra $D$
(unique up to isomorphism), equipped with an involution $\vt$ of the same kind as $\s$.
Then $(A,\s)$ and $(D,\vt)$ are Morita equivalent, which yields a (non-canonical) isomorphism $W_\ve(A,\s) \cong W_{\ve\mu}(D,\vt)$ where $\mu\in\{-1,1\}$. 
For the purpose of this paper (the study of signatures) and without loss of generality
we may assume that $\ve=\mu=1$, cf. \cite[2.1]{AU} and thus only consider hermitian forms over $(A,\s)$.

Let $h:M\x M\to A$ be a hermitian form over $(A,\s)$. By Wedderburn theory, $M$ decomposes into a direct sum of simple right $A$-modules that are all isomorphic to $D^n$, $M\cong (D^n)^k$, for some $k\in \N$ which we call the \df{rank} of $h$, denoted $\rk h$. The rank of $h$ is invariant under Morita equivalence, cf. \cite[2.1]{BP1} or \cite[\S 3.2.1]{GB}.

For $\ell\in \N$, diagonal forms on the free $A$-module $A^\ell$ are denoted by $\la a_1,\ldots, a_\ell\ra_\s$ with $a_1,\ldots, a_\ell \in \Sym(A,\s)^\x$.  We call $\ell$ the \df{dimension} of the form. Note that the rank of this form is equal to $\ell n$ and that its dimension 
may not be invariant under Morita equivalence.
If $A$ is a division algebra, all hermitian forms can be expressed up to isometry in diagonal form and rank and dimension are equal.

Assume now that $F$ is a formally real field with space of orderings $X_F$.
Let $P\in X_F$, let $F_P$ denote a real closure of $F$ at $P$ and consider 
\begin{equation}\label{seq}
\xymatrix{
W(A,\s) \ar[r] & W(A\ox_F F_P, \s\ox\id) \ar[r]^--{\CM} & W_\ve(D_P,\vt_P)\ar[r]^--{\sign_P} & \Z,
}
\end{equation}
where the first map is induced by scalar extension, the second map is the isomorphism
 of $W(F_P)$-modules induced by Morita equivalence and 
\begin{itemize}
\item $\sign_P$ is 
the classical signature isomorphism
 if  $\ve=1$ and
$(D_P, \vt_P) \in \{(F_P, \id),$ $(F_P(\sqrt{-1}), -),$ $((-1,-1)_{F_P}, -)\}$ 
(where $-$ denotes conjugation and quaternion conjugation, respectively);

\item $\sign_P\equiv 0$ if $\ve=-1$ and $(D_P, \vt_P) \in \{(F_P, \id),  ((-1,-1)_{F_P}, -), 
((-1,-1)_{F_P}\x (-1,-1)_{F_P}, \widehat{\phantom{x}}), (F_P\x F_P, \widehat{\phantom{x}})
\}$ (where 
$\widehat{\phantom{x}}$ denotes the exchange involution).
\end{itemize}
We call $\Nil[A,\s]:=\{P\in X_F \mid \sign_P \equiv 0\}$ the set of \emph{nil-orderings} of $(A,\s)$. It depends only on the Brauer class of $A$ and the type of $\s$ (orthogonal, symplectic or unitary). In addition it is clopen in $X_F$ \cite[6.5]{AU}. Let $\Xt_F:= X_F \setminus \Nil[A,\s]$.
We recall the definition of $M$-signatures and $H$-signatures, cf. \cite[3.2, 3.9]{AU}:

\begin{defi} Let $h\in W(A,\s)$,  $P\in X_F$ and $\CM$ be as in \eqref{seq}. The \emph{M-signature} of $h$ at $(P,\CM)$ is defined by $\sign_P^\CM h:=\sign_P\CM(h\ox F_P)$
and is independent of the choice of $F_P$ (by \cite[3.3]{AU}).
\end{defi}

If we choose a different Morita map $\CM'$ in \eqref{seq}, then $\sign_P^{\CM'} h =\pm \sign_P^\CM h$, cf. \cite[3.4]{AU}, which prompts the question if there is a way to make the $M$-signature independent of the choice of Morita equivalence. 

\begin{defi}\label{def:ref-form}
A finite tuple $(h_1,\ldots, h_s)$ of hermitian forms over $(A,\s)$ is called a \emph{tuple of reference forms} if for every $P\in \Xt_F$ and Morita map $\CM$ as in \eqref{seq} there exists $i\in \{1,\ldots, s\}$ such that
$\sign^\CM_P h_i\not=0$.
\end{defi}

It follows from \cite[6.4]{AU} that:

\begin{thm}\label{thm1} 
There exists a tuple $H=(h_1,\ldots, h_s)$ of reference forms over $(A,\s)$ such that  
$h_i$ is a diagonal hermitian form of dimension one  for all $i\in \{1,\ldots, s\}$.
\end{thm}

\begin{cor} There exists a tuple $H=(h_1,\ldots, h_r)$ of reference forms over $(A,\s)$ such that  
$h_i$ is a  hermitian form of rank one  for all $i\in \{1,\ldots, r\}$.
\end{cor}

\begin{proof}
Let $P\in \Xt_F$. Then some form in $W(A,\s)$ has non-zero $H$-signature at $P$. Therefore
some form $h_P$ of rank one has non-zero $H$-signature at $P$, since $W(A,\s)$ is 
generated additively by forms of rank one. We may assume that $\sign^H_P h_P>0$
 and thus
\[\Xt_F= \bigcup_{P\in \Xt_F} \{Q \in \Xt_F \mid \sign^H_Q h_P >0\}=\bigcup_{i=1}^r \{Q \in \Xt_F \mid \sign^H_Q h_i >0\},\]
by compactness of $\Xt_F$ and continuity of the maps $\sign^H h_P$ (cf. \cite[7.2]{AU})  and where the $h_i$ have rank one. 
\end{proof}

\begin{defi}\label{def:s-sign} 
Let $H=(h_1,\ldots, h_s)$ be a tuple of reference forms over $(A,\s)$.
Let $h\in W(A,\s)$ and  $P\in X_F$. We define the \emph{$H$-signature of $h$ at $P$} as  follows:
If $P\in \Nil[A,\s]$, define
$\sign_P^H h:=0$. 
If $P\not\in \Nil[A,\s]$, let $i\in \{1,\ldots, s\}$ be the least integer such that $\sign_P^\CM h_i\not=0$, let $\delta \in \{-1,1\}$ be the sign of $\sign_P^\CM h_i$ and define
\[\sign_P^H h:= \delta \sign_P^\CM h.  \]
\end{defi}

The $H$-signature at $P$, $\sign_P^H :W(A,\s)\to \Z$,
 is independent of the choice of Morita equivalence $\CM$ and is a refinement of the definition of signature in \cite{BP}, the latter not being
defined when $\s$ becomes hyperbolic over $A\ox_F F_P$, cf. \cite[3.11]{AU}. 
It is a morphism of additive groups
 that respects the $W(F)$-module structure of $W(A,\s)$, cf. \cite[3.6]{AU}.

Note that if $h_0$ is such that $\sign_P^H h_0 \not= 0$
  for every $P \in \Xt_F$, then $h_0$ is a reference form and
 $\sign^{(h_0)}_P h_0>0$ for every $P\in \Xt_F$.

We remark that \eqref{def:s-sign} is slightly more general than \cite[3.9]{AU}, 
where the definition was given only for reference tuples consisting of 
diagonal forms of dimension one.  
In the next section we will show that any tuple $H$ of reference forms over $(A,\s)$ can be replaced by just one reference form $h_0$ such that $\sign_P^H h=\sign_P^{(h_0)}h$ for all $h\in W(A,\s)$ and all $P\in X_F$.

We finish this section by 
listing some properties of $H$-signatures: 

\begin{thm}\mbox{}
\begin{enumerate}[$(i)$]
\item Let $h$ be a hyperbolic form over $(A,\s)$, then
$\sign^H_P h=0$.

\item Let $h_1, h_2 \in W(A,\s)$, then
$\sign^H_P (h_1\perp h_2)=\sign^H_P h_1 + \sign^H_P h_2$.

\item Let $h\in W(A,\s)$  and $q\in W(F)$, then
$\sign_P^H (q\cdot h) = \sign_P q \cdot \sign_P^H h$.

\item Let $h\in W(A,\s)$. The total H-signature of $h$, $\sign^H h:X_F\to \Z, P\mapsto \sign_P^H h$, is continuous.

\item \textup{(}Pfister's local-global principle\textup{)} Let $h\in W(A,\s)$. Then $h$ is a torsion form if and only if $\sign^H_P h=0$ for all $P\in X_F$.

\item \textup{(}Going-up\textup{)} Let $h\in W(A,\s)$ and let $L/F$ be an algebraic extension of ordered fields. Then 
\[\sign_Q^{H\ox L} (h\ox L)=\sign_{Q\cap F}^H h\] 
for all $Q\in X_L$. 

\item \textup{(}Going-down: Knebusch trace formula\textup{)} 
Let $L/F$ be a finite extension of ordered fields and assume $P\in X_F$ extends to $L$.
Let $h\in W(A\ox_F L, \s\ox\id)$. Then
\[\sign_P^H(\mathrm{Tr}^*_{A\ox_F L} h) = \sum_{P\subseteq Q \in X_L} \sign_Q^{H\ox L} h,\]
 where $\mathrm{Tr}^*_{A\ox_F L} h$ denotes the Scharlau transfer induced by the $A$-linear homomorphism $\id_A \ox \mathrm{Tr}_{L/F}: A\ox_F L\to A$.
\end{enumerate}
\end{thm}

\begin{proof}  See \cite[4.1]{LU} for $(v)$ and \cite[3.6, 7.2, 8.1]{AU} for the other statements, except for $(vi)$ which is contained implicitly in \cite{AU}. For $(vi)$ it suffices to provide an argument for 
$M$-signatures: let $h\in W(A,\s)$ and let $P=Q\cap F$, then
\[\sign_Q^\CM (h\ox L)= \sign \CM \bigl( (h\ox L)\ox L_Q\bigr) = \sign \CM (h\ox L_Q)
=\sign_P^\CM h,\]
where $L_Q$ is a real closure of $L$ at $Q$ and therefore a real closure of $F$ at $P$.
\end{proof}

\section{A Reference Tuple of Length One}

We equip $X_F$ with the Harrison topology and $\Z$ and $\{-1,1\}$ each with the discrete topology. 

\begin{lemma}\label{induction}
Let $H$ be any tuple of reference forms over $(A,\s)$, let
$h_1, \ldots, h_\ell$ be hermitian forms over $(A, \sigma)$ and let
\[V=  \{P \in X_F \mid \exists i\in \{1,\ldots, \ell\}\text{ such that } \sign_P^H h_i\not=0\}.\]
Then there exists a hermitian form $h$ over $(A, \sigma)$ such that 
\[V=\{P \in X_F \mid  \sign_P^H h\not=0\}.\]
\end{lemma}

\begin{proof}
  By induction on $\ell$. The case $\ell=1$ is clear. Assume that $\ell>1$ and let
\[  V' = \{P \in X_F \mid \exists i\in \{1,\ldots, \ell-1\}\text{ such that } \sign_P^H h_i\not=0\}.\]
  By induction there exists a hermitian form $h'$ such that 
  $V'= \{P \in X_F \mid  \sign_P^H h'\not=0\}$. Note that $V'$ is clopen in $X_F$ since the total
  signature map $\sign^H h'$ is continuous \cite[7.2]{AU}. Thus,
  by \cite[VIII, 6.10]{Lam} there exists a quadratic form $q$ over $F$
  such that $\{P \in X_F \mid  \sign_P q=0\} = V'$. Let
  \[h := h' + q \cdot h_{\ell}.\]
  Let $P \in V$. We have  to consider two cases:
  \begin{itemize}
    \item $P \in V'$. Then $\sign_P^H h' \not = 0$ while $\sign_P q =0$. Thus
      $\sign_P^H h =\sign_P^H h'+ (\sign_P q)(\sign_P^H h_\ell) \not = 0$.
    \item $P \in V \setminus V'$. Then $\sign_P^H h' = 0$ and $\sign_P q \not =
      0$. We also have $\sign_P^H h_{\ell} \not = 0$ by definition of $V$.
      Therefore $\sign_P^H h \not = 0$.
  \end{itemize}
We conclude that   $V\subseteq \{P \in X_F \mid  \sign_P^H h\not=0\}$. For the reverse inclusion, let $P\in X_F$ be such that $\sign_P^H h\not=0$. If $\sign_P^H h'\not=0$, then $P\in V'\subseteq V$.
If $\sign_P^H h'=0$  then $\sign_P^H(q\cdot  h_{\ell})\not=0$, which implies $\sign_P^H h_{\ell}\not=0$ and thus $P\in V$.
\end{proof}

\begin{prop}\label{single-form} Let $H = (h_1, \ldots, h_s)$ be a tuple of reference forms over $(A,\s)$.
There exists a hermitian form $h_0$ over $(A,\sigma)$ such that $\sign_P^H h_0 > 0$
  for every $P \in \Xt_F$.
\end{prop}

\begin{proof} For every $P \in \Xt_F$ one of  
$\sign_P^H h_1, \ldots, \sign_P^H h_s$ is non-zero.
Therefore
\[\Xt_F=\{P \in X_F \mid \exists i\in \{1,\ldots, s\}\text{ such that } \sign_P^H h_i\not=0\}\]
and  by \eqref{induction}  there exists a hermitian form $h'$ over $(A,\s)$ such that 
$\sign_P^H h' \not= 0$  for every $P \in \Xt_F$.
Consider the disjoint clopen subsets $S:=\{P\in X_F \mid \sign^H_P h' >0\} \cup \Nil[A,\s]$ 
and $T:=\{P\in X_F \mid \sign^H_P h' <0\} $ of $X_F$. The function $f:X_F \to \Z$ defined by $f=1$ on $S$ and $f=-1$ on $T$ is continuous, so by
\cite[VIII, 6.9]{Lam}  there exists a quadratic form $q$ over $F$
such that $\sign q>0$ on $S$ and
$\sign q<0$ on $T$. Let $h_0=q\cdot h'$. Then
 $\sign^H_Ph_0 >0$ for every $P\in \Xt_F$.
 \end{proof}

The forms $h_1,\ldots, h_s$ in \eqref{thm1} are diagonal.
Thus the form $h_0$ in \eqref{single-form}
can also be chosen to be diagonal, which can be seen 
from the proof of \eqref{induction}.

\begin{prop}\label{h_zero}
Let $H$ be a tuple of reference forms over $(A,\s)$.
\begin{enumerate}[$(i)$]
\item Let $h_0$ be as in \eqref{single-form}. Then  $\sign^H =
\sign^{(h_0)}$.

\item Let $f$ be a continuous map from $X_F$ to $\{-1,1\}$.
 There exists $h_f \in W(A,\s)$ such that
$\sign^{(h_f)} = f \cdot \sign^H$.

\item Let $H'$ be another  tuple of reference forms over $(A,\s)$. There exists a 
continuous map $f$ from $X_F$ to $\{-1,1\}$ such that
$\sign^{H'} = f \cdot \sign^H$.
\end{enumerate}
\end{prop}

\begin{proof} 
$(i)$ By \eqref{single-form} there exists  $h_0 \in W(A,\s)$ such that $\sign_P^H 
h_0 >0$ for all $P\in \Xt_F$.
Let $h\in W(A,\s)$.
  Considering that, for a given $P \in X_F$, $\sign^H_P$ is a special case of
  $\sign^{\CM}_P$, the definition of $\sign^{(h_0)}_P h$ given in \eqref{def:s-sign} 
  can be expressed as follows:
  \[\sign^{(h_0)}_P h =
      \begin{cases}
        \sign^H_P h & \text{if } P\in \Xt_F,\\
        0 & \text{if } P\in \Nil[A,\s],
      \end{cases}\]
which proves the result.

$(ii)$ By $(i)$ there exists $h_0 \in W(A,\s)$ such that $\sign^{(h_0)} = \sign^H$
with $\sign^H_P h_0 > 0$ for every $P \in \Xt_F$. 
By \cite[VIII, 6.9]{Lam} there exists a quadratic form $q$ over $F$
such that $\sign q > 0$ on $f^{-1}(1)$ and
$\sign q < 0$ on $f^{-1}(-1)$. 
Let $h_f = q\cdot h_0$.   Then for $h \in W(A,\s)$ we have
\begin{align*}
    \sign^{(h_f)}_P h & = 
      \begin{cases}
        \sign^H_P h & \text{if } \sign_P^H(q\cdot h_0) > 0, \text{ i.e. if } f(P) = 1\\
        -\sign^H_P h & \text{if } \sign_P^H(q\cdot h_0) < 0, \text{ i.e. if } f(P) = -1\\
      \end{cases}\\
    & = f(P) \sign^H_P(h).
\end{align*}

$(iii)$ Let $P \in \Xt_F$. By \cite[3.4]{AU} there exists $\delta \in \{-1,1\}$ such that
$\sign_P^H = \delta \sign_P^{H'}$. Let $h_0 \in W(A,\s)$ be such that $\sign^H = \sign^{(h_0)}$. Then $\sign_P^H h_0 >0$ and $\sign_P^H h_0= \delta \sign_P^{H'} h_0$, so $\delta = \sgn\bigl(\sign_P^{H'}(h_0)\bigr)$, where $\sgn$ denotes the sign function. Define $f: X_F \to \{-1,1\}$ by
 \[f(P) =
      \begin{cases}
        \sgn\bigl(\sign_P^{H'}(h_0)\bigr) & \text{if } P\in \Xt_F,\\
        1 & \text{if } P\in \Nil[A,\s].
      \end{cases}\]
Then $f$ is continuous and $\sign^H = f\cdot \sign^{H'}$.      
\end{proof}

\section{Behaviour under Morita Equivalence}

\begin{lemma}\label{Meq-lemma}
Let $(A,\s)$ and $(B,\tau)$ be $F$-algebras with involution.
Let $L/F$ be a field extension. Any Morita equivalence of  $(A,\s)$ and $(B,\tau)$ extends to a Morita equivalence
of $(A\ox_F L, \s\ox \id)$ and $(B\ox_F L, \tau \ox\id)$ and the induced diagram of Witt groups \textup{(}i.e. where the horizontal maps are the
isomorphisms induced by the Morita equivalences and the vertical maps  are the canonical scalar extension maps\textup{)}
\begin{equation}\label{diag1}
\begin{split}
\xymatrix{
W_\ve(A\ox_F L, \s\ox\id)\ar[r]^--{\sim} & W_{\ve\mu}(B\ox_F L, \tau\ox\id)\\
W_\ve(A,\s) \ar[r]^--{\sim} \ar[u] & W_{\ve\mu}(B,\tau) \ar[u]
}
\end{split}
\end{equation}
commutes. Here $\mu=-1$ if $\s$ and $\tau$ are of the first kind and opposite type and $\mu=1$ otherwise.
\end{lemma}

\begin{proof}
Let $((A,\s), (B,\tau), M,N, f, g, \nu)$ be a tuple of Morita equivalence data. In other words (cf. \cite{D} and 
\cite[Chap.~2]{GB}), $M$ is an $A$-$B$ bimodule, $N$ a $B$-$A$ bimodule, 
$f:M\ox_B N\to A$ and $g: N\ox_A M \to B$ are bimodule homomorphisms (in fact, isomorphisms; cf. \cite{D} and 
\cite[Chap.~2]{GB}) that are associative in the sense that
\begin{align*}
f(m\ox n)\cdot m'&= m\cdot g(n\ox m'),\\
g(n\ox m)\cdot n' &= n\cdot f(m\ox n')
\end{align*}
for all $m,m' \in M$, $n,n'\in N$ and $\nu: M\to N$ is an $F$-linear bijective map such that $\nu(amb)=\tau(b)\nu(m)\s(a)$, for all
$a\in A, b\in B, m\in M$.  By \cite[1.2]{D}, $N$ is isomorphic to $\Hom_B(M,B)$, and we must also require that 
 the map $M\x M\to B, (x,y)\mapsto [\nu(x)](y)$ is a hermitian form over $(B,\tau)$, as pointed out in \cite[\S1.3, \S1.5]{CL}.

Consider the tuple $((A\ox_F L, \s\ox\id), (B\ox_F L, \tau\ox\id), M\ox_F L,
N\ox_F L, f_L, g_L, \nu_L)$ canonically induced by scalar extension from $F$ to $L$. 
It is not difficult to verify that it is again a tuple of Morita equivalence data.

We then have isomorphisms of $F$-algebras with involution (cf. \cite{D} and 
\cite[Chap.~2]{GB}) 
$(A,\s) \cong (\End_B(M), \ad_{b_0})$ and $(A\ox_F L,\s\ox\id) \cong (\End_{B\ox L}(M\ox_F L), \ad_{{b_0}'})$,
where
\[b_0: M\x M\too B, (m,m')\mapstoo g(\nu(m)\ox m')\]
and
\[{b_0}': (M\ox_F L)\x (M\ox_F L)\too B\ox_F L, (w,w')\mapstoo g_L(\nu_L(w)\ox w')\]
are $\mu$-hermitian forms over $(B,\tau)$ and $(B\ox_F L, \tau\ox\id)$, respectively,  and where  in fact ${b_0}'=b_0\ox L$, as can easily be verified. 
The horizontal isomorphisms in \eqref{diag1} are given by
\[W_\ve(A,\s) \simtoo W_{\ve\mu}(B,\tau), (P,b)\mapstoo (P\ox_A M, b_0b) \]
and
\[W_\ve(A\ox_F L,\s\ox \id) \simtoo W_{\ve\mu}(B\ox_F L,\tau\ox\id ), 
(P',b')\mapstoo (P'\ox_{A\ox L} (M\ox_F L), {b_0}'b'),\]
where 
\[b_0b(p\ox m, p'\ox m') := b_0(m, b(p,p') m')\quad \text{for all } m, m'\in M, p, p' \in P\]
and 
\[{b_0}'b'(p\ox w, p'\ox w') := {b_0}'(w, b'(p,p') w')\quad \text{for all } w, w'\in M\ox_F L, p, p' \in P'.\]
Now let $(P,b) \in W_\ve(A,\s)$. It follows from a straightforward computation that
\[\bigl( (P\ox_A M)\ox_F L, (b_0b)\ox L\bigr) = \bigl( (P\ox_F L) \ox_{A\ox L} (M\ox_F L), {b_0}'(b\ox L)\bigr)\]
and we conclude that the diagram \eqref{diag1}  commutes.
\end{proof}

\begin{thm} Let $(A,\s)$ and $(B,\tau)$ be Morita equivalent $F$-algebras with involution and assume that $\s$ and $\tau$ are of the same kind and type. Let $\zeta: W(A,\s)\simtoo W(B,\tau)$ be the induced isomorphism of Witt groups. Then
\[\sign_P^H h = \sign_P^{\zeta(H)} \zeta(h)\]
for all $h\in W(A,\s)$ and all $P\in X_F$.
\end{thm}

\begin{proof}
If $P\in \Nil[A,\s]=\Nil[B, \tau]$ there is nothing to prove. Thus assume that $P\in  \Xt_F$.
By  \eqref{h_zero} we may assume that $H=(h_0)$. Then \eqref{Meq-lemma} gives
a Morita equivalence of
 $(A_P,\s_P):=(A\ox_F F_P, \s\ox\id)$ and $(B_P, \tau_P):=(B\ox_F F_P, 
\tau\ox\id)$ such that the induced isomorphism
$\xi: W(A_P, \s_P) \simtoo W(B_P, \tau_P)$
makes the diagram
\begin{equation*}
\begin{split}
\xymatrix{
W(A_P, \s_P)\ar[r]^--{\xi}_{\sim} & W(B_P, \tau_P)\\
W(A,\s) \ar[r]^--{\zeta}_{\sim} \ar[u] & W(B,\tau) \ar[u] 
}
\end{split}
\end{equation*}
(where the vertical arrows are the canonical scalar extension maps) commute.
Now $(A_P,\s_P)$ is Morita equivalent to an $F$-algebra with involution $(D_P,\vt_P)$ as in \eqref{seq} (with $\ve=1$ since $P\not\in \Nil[A,\s]$) and we have an
induced isomorphism
\[\CM_1: W(A_P, \s_P) \simtoo W(D_P, \vt_P).\] 
Hence $(B_P,\tau_P)$ is Morita equivalent to 
$(D_P,\vt_P)$ via the equivalences that induce $\xi^{-1}$ and $\CM_1$ and we obtain an induced isomorphism
\[\CM_2:=\CM_1\circ \xi^{-1}: W(B_P, \tau_P) \simtoo W(D_P, \vt_P).\]
It follows that the diagram 
\begin{equation*}
\begin{split}
\xymatrix{
W(A, \s)\ar[d]_--{\zeta}  \ar[r] & W(A_P, \s_P)\ar[d]_--{\xi} 
\ar[r]^--{\CM_1} &  W(D_P, \vt_P)\ar[r]^--{\sign_P} & \Z\\
W(B, \tau) \ar[r] & W(B_P,\tau_P)\ar[ru]_--{\CM_2}  & &
}
\end{split}
\end{equation*}
commutes. Let $h\in W(A,\s)$. Then 
\begin{align*}
\sign_P^H h &= \sgn\bigl( \sign_P \CM_1(h_0\ox F_P)\bigr)\cdot \sign_P \CM_1 (h\ox F_P)\\
&= \sgn\bigl( \sign_P \CM_2 (\zeta(h_0)\ox F_P)\bigr)\cdot \sign_P \CM_2 (\zeta(h)\ox F_P)\\
&= \sign_P^{\zeta(H)} \zeta(h). \qedhere
\end{align*}
\end{proof}

\section{m-Ideals of Modules}

Let $(A,\s)$ be an $F$-algebra with involution.
We want to extend to $W(A,\s)$
 the well-known classification of all prime ideals of $W(F)$ in
terms of the kernels of the signature maps. We take some moments to look at the
properties of $\ker \sign_P^H$ for $P \in X_F$, inside $W(A,\s)$ considered as a
$W(F)$-module. It follows from \cite[3.6]{AU} that

\begin{enumerate}
  \item $(\ker \sign_P) \cdot W(A,\s) \subseteq \ker \sign_P^H$;
  \item $W(F) \cdot \ker \sign_P^H \subseteq \ker \sign_P^H$;
  \item for every $q\in W(F)$ and every $h\in W(A,\s)$, $q\cdot h \in \ker \sign_P^H$ implies
   that $q \in \ker \sign_P$ or $h \in \ker \sign_P^H$.
\end{enumerate}
These properties motivate the following

\begin{defi}
Let $R$ be a commutative ring and let $M$ be an $R$-module.
An \emph{m-ideal} of $M$ is a pair $(I,N)$ where $I$ is an ideal of $R$, $N$
      is a submodule of $M$, and such that $I \cdot M \subseteq N$.

An m-ideal $(I,N)$ of $M$ is \emph{prime} if $I$ is a prime ideal of $R$ (we
      assume that all prime ideals are proper), $N$ is a proper submodule of
      $M$, and         
      for every $r \in R$ and $m \in M$, $r\cdot m \in N$ implies that $r \in I$
      or $m \in N$.
\end{defi}

\begin{ex} 
The pair $(\ker \sign_P, \ker \sign_P^H)$ is a prime m-ideal of the $W(F)$-module $W(A,\s)$
whenever $P \in X_F \setminus \Nil[A,\s]$.
\end{ex}

The following lemma is immediate.

\begin{lemma} 
Let $(I,N)$ be a prime m-ideal of the $R$-module $M$. Then $I=\{r\in R \mid rM\subseteq N\}$.
\end{lemma}

\begin{defi}
  Let $R$ and $S$ be commutative rings, let $M$ be an $R$-module and $N$  an
  $S$-module. We say that a pair $(\vf, \psi)$ is an \df{$(R,S)$-morphism 
  \textup{(}of modules\textup{)} from 
  $M$ to $N$}
  if 
  \begin{enumerate}
    \item $\vf : R \rightarrow S$ is a morphism of rings (and in particular $\vf(1)=1$);
    \item $\psi : M \rightarrow N$ is a morphism of additive groups;
    \item for every $r \in R$ and $m \in M$, $\psi(r\cdot m) = \vf(r)\cdot \psi(m)$.
  \end{enumerate}
  We call an $(R,S)$-morphism $(\vf,\psi)$  \df{trivial} if $\psi=0$. 
  We denote  the set of all $(R,S)$-morphisms from $M$ to
  $N$ by $\Hom_{(R,S)}(M,N)$ and its subset of non-trivial $(R,S)$-morph\-isms by $\Hom^*_{(R,S)}(M,N)$.
\end{defi}

\begin{ex} The pair
$(\sign_P, \sign_P^H)$ is a $(W(F),\Z)$-morphism from
    $W(A,\s)$ to $\Z$ and is trivial if and only if $P\in \Nil[A,\s]$.
\end{ex}

The following three lemmas are immediate.
\begin{lemma}\label{m-ideal}
Let $(\vf, \psi)$  be an $(R,S)$-morphism from $M$ to $N$. Then
$(\ker \vf, \ker \psi)$ is an m-ideal of the $R$-module $M$. Furthermore,
$(\ker \vf, \ker \psi)$ is prime if and only if the $\vf(R)$-module $\psi(M)$ is torsion-free.
\end{lemma}

\begin{lemma}\label{torsion-free}
Let $(I,N)$ be an m-ideal of the $R$-module $M$. Then 
  $M/N$ is canonically an $R/I$-module, where the product is defined
  by $(r+I) \cdot (m+N) = r\cdot m + N$. 
\end{lemma}

\begin{lemma}\label{tf}
Let $(I,N)$ be an m-ideal of the $R$-module $M$, and let $\vf : R \rightarrow R/I$ and 
$\psi  : M \rightarrow M/N$ be the canonical projections. Then
$(\vf, \psi)$ is an $(R,R/I)$-morphism of modules. Furthermore,
$(I,N)$ is prime if and only if $R/I\not=\{0\}$ and $M/N$ is a non-zero torsion-free $R/I$-module.
\end{lemma}

\section{Classification of Prime m-Ideals of $W(A,\s)$}

Let $(I,N)$ be a prime m-ideal of the $W(F)$-module $W(A,\s)$. By classical results 
of Harrison \cite{H} and Lorenz and Leicht \cite{LL}
on $W(F)$, cf. \cite[VIII, 7.5]{Lam},
there are only three possibilities for $I$: 
\begin{enumerate}
  \item $I = \ker \sign_P$ for some $P \in X_F$;
  \item $I = \ker (\pi_p \circ \sign_P)$ for some $P \in X_F$ and prime $p>2$, 
     where $\pi_p : \Z \rightarrow \Z/p\Z$ is the canonical
    projection;
  \item $I = I(F)$, in which case $I = \ker (\pi_2 \circ \sign_P)$ for any $P \in
    X_F$.
\end{enumerate}
In this section we will investigate in how far $I$ determines $N$. In the first two cases (when $2 \not \in I$)  $N$ is uniquely determined by $I$,  but in the third case (when $2\in I$) this is not so.
We prove three auxiliary results first.

\begin{lemma}\label{tors}
Let $P\in X_F$ and let $h\in W(A,\s)$ be such that $\sign_P^H h=0$. Then there exist $n\in \N_0:=\N\cup\{0\}$ and $q\in W(F)$ such that $\sign_P q=2^n$ and $q\cdot h\in W(A,\s)_t$,  where $W(A,\s)_t$ denotes the torsion subgroup of $W(A,\s)$.
\end{lemma}

\begin{proof}
Let
  $U = (\sign^H h)^{-1}(0)$. The set $U$ is clopen in $X_F$ since the total signature map
  $\sign^Hh$ is continuous \cite[7.2]{AU}.  Hence the function $\chi_U :
  X_F \rightarrow \Z$ defined by $\chi_U = 1$ on $U$ and $\chi_U = 0$ on $X_F \setminus U$
  is continuous and there are $n \in \N_0$ and $q \in W(F)$ such that $\sign q
  = 2^n \chi_U$, cf. \cite[VIII, 6.10]{Lam}. It follows that $\sign^H(q\cdot h) = \sign(q) \sign^H(h) = 0$ on $X_F$, and so $q\cdot h \in
  W(A,\s)_t$, by \cite[4.1]{LU}.
\end{proof}

\begin{lemma}\label{kernels} 
Let $(f,g) \in \Hom_{(W(F),\Z)} (W(A,\s),\Z)$. Then there exists a unique $P\in X_F$ such that
$f=\sign_P$ and $\ker \sign_P^H \subseteq \ker g$. In particular, if $P\in\Nil[A,\s]$, then $g=0$.
\end{lemma}

\begin{proof}
The morphism of rings $f : W(F) \to \Z$ is completely determined by $P:=\{a \in F^\x \mid f(\la a\ra)=1\}\cup\{0\}$, which is an ordering of $F$ (follow for instance the proof of \cite[VIII, 3.4(2)]{Lam}). Therefore $f=\sign_P$.

Let $h \in \ker \sign_P^H$. By \eqref{tors} there exists  $n\in \N_0$ and $q\in W(F)$ such that $\sign_P q=2^n$ and $q\cdot h\in W(A,\s)_t$. Since $f =\sign_P$, we obtain $f(q)\not=0$ and since $g$ has values in $\Z$ (which is torsion-free), we have $0=g(q\cdot h)=f(q) \cdot g(h)$, which implies
  $g(h) = 0$. Therefore $\ker \sign^H_P \subseteq \ker g$.
\end{proof}

\begin{lemma}\label{2inI}
Let   $(I, N)$ be an m-ideal of $W(A,\s)$ with $N\not=W(A,\s)$ and with $I=\ker(\pi_p\circ \sign_P)$ for some $P\in X_F$ and prime $p$.
Then $(I,N)$ is prime.
\end{lemma}

\begin{proof} 
Since $I$ is a prime ideal of $W(F)$ we only need to check that for $q\in W(F)$ and $h\in W(A,\s)$ such that $q\cdot h\in N$ we have
$q\in I$ or $h\in N$. Assume that $q\not\in I$. Since $W(F)/I$ is a field, $q$ is invertible modulo $I$, i.e. there exists $q'\in W(F)$ such that
$qq'=1+i$ with $i\in I$. Hence the assumption $q\cdot h\in N$ implies $qq'\cdot h=(1+i)\cdot h \in N$ and thus $h\in N$ since $i\cdot h\in I\cdot W(A,\s)\subseteq N$.
\end{proof}

\subsection*{The case $\bs{2 \not \in I}$}

\begin{lemma}\label{torsion}
  Let $(I,N)$ be a prime m-ideal of the $W(F)$-module $W(A,\s)$ such that $2 \not \in I$.   Then
  $W(A,\s)_t \subseteq N$.  
\end{lemma}

\begin{proof}
  Let $h \in W(A,\s)_t$. Then there exists $n \in \N$ such that $0= 2^n\cdot h$ by \cite[5.1]{Sch1}. 
  Since $0 \in N$, this implies $2^n \in I$ or $h \in N$. 
  Since $2 \not
  \in I$, we obtain $h \in N$.
\end{proof}

\begin{prop}\label{classif}
  Let $(I,N)$ be a prime m-ideal of the $W(F)$-module $W(A,\s)$ with $2 \not \in I$. Then one of
  the following holds:
  \begin{enumerate}[$(i)$]
    \item There exists   $P \in X_F$ such that $(I,N) = (\ker \sign_P,  \ker \sign_P^H)$.
    \item There exist   $P \in X_F$ and a prime $p >2$   such that
    \begin{align*}
     (I,N) &= \bigl(\ker (\pi_p \circ \sign_P), \ker (\pi \circ \sign_P^H) \bigr) \\  
     &=  \bigl(p\cdot W(F) + \ker \sign_P, p\cdot W(A,\s)+\ker \sign_P^H\bigr), 
     \end{align*}
     where
      $\pi: \im \sign_P^H \to \im \sign_P^H/(p\cdot \im\sign_P^H)$ is the canonical projection.
  \end{enumerate}
\end{prop}

\begin{proof} 
We first observe that the ordering  $P$  we are looking for cannot be in $\Nil[A,\s]$, since $N$ is
a proper submodule of $W(A,\s)$.

  We denote by $\vf : W(F) \rightarrow W(F)/I$ and $\psi : W(A,\s) \rightarrow
  W(A,\s)/N$ the canonical projections. By  \eqref{tf}, $\im \psi$ is a
  torsion-free $\im \vf$-module, where the product is given by $\vf(q)\cdot \psi(h) =
  \psi(q\cdot h)$ for $q\in W(F)$ and $h\in W(A,\s)$.
  
  We consider the different possibilities for $I = \ker \vf$. By the classification of
  prime ideals of $W(F)$ there is a $P \in X_F$ such that $\ker \sign_P \subseteq
  I$. Let $h \in \ker \sign_P^H$. By \eqref{tors}  there exists $n\in \N_0$ and $q\in W(F)$ such that $\sign_P q=2^n$ and $q\cdot h\in W(A,\s)_t$. Since $W(A,\s)_t \subseteq N$ by \eqref{torsion} we obtain $0=\psi(q\cdot h)= \vf(q)\cdot \psi(h)$. 
 But we know that $I
  = \ker \sign_P$ or $I = \ker (\pi_p \circ \sign_P)$ for some prime $p>2$. Since 
 $\sign_P(q)=2^n$ and $p>2$,
  we obtain that $q \not \in I = \ker \vf$. Thus $\vf(q)\cdot \psi(h) = 0$ implies
  $\psi(h) = 0$ by  \eqref{tf}. Therefore $\ker \sign^H_P \subseteq N$.
  
 We now consider the two possibilities for $\vf$:
  
$(i)$ $\ker \vf = \ker \sign_P$. Then $\im \vf \cong \Z$ is a torsion-free abelian
      group, and it follows from \eqref{tf} that $\im \psi$ is a
      torsion-free abelian group. Using $\ker \sign^H_P \subseteq \ker \psi=N$, we
      obtain
      \[\im \psi = W(A,\s)/N \cong \bigl(W(A,\s)/\ker \sign^H_P\bigr)/\bigl(N/\ker
      \sign^H_P\bigr),\]
      where $W(A,\s)/\ker \sign^H_P\cong \im \sign_P^H$ is an infinite cyclic group.       
      Since $\im \psi$ is torsion-free, the only
      possibility  is $N/\ker \sign^H_P = \{0\}$,
      i.e. $N = \ker \sign_P^H$.

$(ii)$ $\ker \vf = \ker (\pi_p \circ \sign_P)$. Then $\im \vf$ is a field with $p$
elements and $\im \psi = W(A,\s)/N$ is an $\im \vf$-vector space. In particular, for every $h\in W(A,\s)$, we have that $p\cdot h\in N$.
Since we also have $\ker \sign_P^H\subseteq N$ we obtain
$N_0:=p\cdot W(A,\s)+\ker\sign_P^H \subseteq N $.     

Consider an arbitrary element $p\cdot h+h_0\in N_0$ with $h\in W(A,\s)$ and $h_0\in \ker \sign_P^H$. Then $\sign_P^H(p\cdot h+h_0)=p\sign_P^H h$ and $\sign_P^H$ induces an
isomorphism $N_0/\ker\sign_P^H \cong p\cdot C$, where $C=\im \sign_P^H$ is an infinite cyclic group. 
It follows that 
\[W(A,\s)/N_0\cong \bigl(W(A,\s)/\ker \sign^H_P\bigr)/\bigl(N_0/\ker \sign^H_P\bigr)\cong C/(p\cdot C),\] 
and so $[W(A,\s):N_0]=p$.     

Since $N/\ker \sign_P^H$ is a subgroup of $W(A,\s)/\ker \sign_P^H\cong C$, there must exist 
$k\in \N$ such that $N/\ker \sign_P^H\cong k\cdot C$.
It follows that 
\[W(A,\s)/N\cong \bigl(W(A,\s)/\ker \sign^H_P\bigr)/\bigl(N/\ker \sign^H_P\bigr)\cong C/(k\cdot C),\] 
which we know to be a vector space over a field with $p$ elements. Thus every non-zero
element of $W(A,\s)/N$ has order $p$ and so
$px\in k\cdot C$ for all $x\in C$. It follows that $k\vert p$. If $k=1$, then $N=W(A,\s)$, a contradiction. Thus $k=p$ and so $[W(A,\s):N]=p$.

We conclude that $N=N_0=p\cdot W(A,\s)+\ker\sign_P^H$. 

Finally, consider the canonical projection $\pi: \im \sign_P^H \to \im \sign_P^H/
(p\cdot\im\sign_P^H)$. It is not difficult to see that
$p\cdot W(A,\s)+\ker\sign_P^H=\ker(\pi\circ \sign_P^H)$.
\end{proof}

\subsection*{The case $\bs{2 \in I}$.}

Consider a prime m-ideal  $(I,N)$ of the $W(F)$-module $W(A,\s)$ with $2 \in I$. As observed
above we have $I = I(F)$, the fundamental ideal of $W(F)$. 

We define $I(A,\s)$ to be the submodule of all equivalence classes in
$W(A,\s)$ of forms of even rank, cf. \cite[2.2]{BP1} or \cite[\S 3.2.1]{GB}.
Observe that if $A$ is a division algebra, then $I(A,\s)$ is additively
generated by the forms $\la 1, a \ra_\s$ for $a \in \Sym(A,\s)^\times$ since every form can be
diagonalized. Recall from \cite[\S 3.2.1]{GB}:

\begin{lemma}\label{index2}
  $I(A,\s)$ has index $2$ in $W(A,\s)$.
\end{lemma}

\begin{proof} Since the rank is invariant under Morita equivalence, we may assume that $A$ is a division algebra, in which case
the result follows immediately from the fact that $\la a \ra_\s = \la b \ra_\s\mod I(A,\s)$ for every $a,b \in \Sym(A,\s)^\times$.
\end{proof}

\begin{prop}
A pair $(I(F), N)$ is a prime m-ideal of $W(A,\s)$ if and only if $N$ is a proper submodule of $W(A,\s)$ with 
$I(F)\cdot W(A,\s)\subseteq N$.
\end{prop}

\begin{proof}
Take $p=2$ in \eqref{2inI}.
\end{proof}

We thus see that in contrast to the $2\not\in I$ case, $N$ is not uniquely determined by $I$, 
since there are in general several proper submodules $N$ of $W(A,\s)$ containing
$I(F) \cdot W(A,\s)$. Obviously one such submodule is $I(F) \cdot W(A,\s)$
itself, and $I(A,\s)$ is another one. The following example shows that in general these 
modules are distinct.

\begin{ex}
  Let $F = \R(\!(x)\!)(\!(y)\!)$ be the iterated Laurent series field in the unknowns $x$ and $y$ over the field of real numbers.
  Let $D=(x,y)_F$ be the quaternion algebra with $F$-basis $(1,i,j,k)$ such that
  $i^2=x$, $j^2=y$ and $ij=k=-ji$. Let $\s$ be the orthogonal involution on $D$ that sends $i$ to $-i$ and that fixes the other basis elements. 
  Let $v : F  \rightarrow \Z \times \Z$ be the
  standard $(x,y)$-adic valuation on $F$ (see for instance \cite[Chapter
  3]{Wad}). Note that $F$ is Henselian with respect to $v$ and has residue field  $\R$. An application of
  Springer's theorem shows that the norm form  $\la 1,-x,-y,xy \ra$ of $D$ is
  anisotropic (we obtain four residue forms of dimension $1$ over $\R$, that are
  necessarily anisotropic). Hence $D$ is a division algebra, cf. \cite[III,
  4.8]{Lam}.

  Since $F$ is Henselian, the valuation $v$ extends uniquely to a valuation on
  $D$ (see \cite[Thm. 2]{Mo}), which we also denote by $v$.  We use standard notation with $\bar{F}$ and
  $\bar{D}$ standing for the residue field of $F$ and the residue division algebra of $D$, respectively, and with
  $\Gamma_F:=v(F^\x)$ and $\Gamma_D:=v(D^\x)$.
  
  We now claim that
   $\bar{D}$ is isomorphic to $\R$. The proof of
  this claim goes as follows: From the definition of $v$, it is clear that $\Gamma_D$ contains 
  $ \frac{1}{2} \Z \times \frac{1}{2} \Z$, hence $[\Gamma_D : \Gamma_F] \geq 4$.
  But $[D:F]=4$ and it readily follows from  Draxl's ``Ostrowski Theorem'' (see
  \cite[Equation 1.2]{j-w}) that $\bar{D}=\bar{F}$ (and $[\Gamma_D : \Gamma_F] = 4$, so
  $\Gamma_D=\frac{1}{2} \Z \times \frac{1}{2} \Z$).

  Now a system of representatives of the set  $v(\Sym(D,\s)^\x)$ in $\Gamma_D/2\Gamma_D$ is given by $\Bigl\{(0,0),
  (0,\frac{1}{2}), (\frac{1}{2}, \frac{1}{2})\Bigr\}$, corresponding to the set of
  uniformizers $\{1, j, k\}$, and by \cite[Thm. 3.7]{Larmour} we obtain that
  \begin{equation}\label{star}
  \begin{split}
    W(D,\s) &\cong W(\bar D,\id) \oplus W(\bar D, \id) \oplus W(\bar D, \id)\\
            &\cong W(\R) \oplus W(\R) \oplus W(\R)\\
            &\cong \Z \oplus \Z \oplus \Z,
  \end{split} 
  \end{equation}
  where the first isomorphim is obtained by computing the residue forms
  corresponding to the three uniformizers. 
  We use this description of $W(D,\s)$
  to compute $I(F) \cdot W(D,\s)$. Let $\lambda$ denote the isomorphism  $W(D,\s)\cong \Z \oplus \Z \oplus \Z$
  given by \eqref{star}.
  Since $I(F)$ is additively generated by the
  forms $\la 1, a \ra$ for $a \in F^\times$ and $W(D,\s)$ is additively generated
  by the diagonal forms of dimension one (since $D$ is division), we only have to consider
  the forms $\la 1,a \ra \cdot \la \alpha \ra_\s = \la \alpha, a \alpha
  \ra_\s$ for $a \in F^\times$ and $\alpha \in \Sym(D,\s)^\times$. Since $v(a) \in
  \Gamma_F = \Z \times \Z$ we have $v(a\alpha) = v(a) + v(\alpha) \in v(\alpha)
  + 2 \Gamma_D = v(\alpha) + \Z \times \Z$. It follows that the form $\la \alpha, a
  \alpha\ra_\s$ only gives rise to one residue form in $W(\bar D, \id) \cong
  W(\R) \cong \Z$. This residue form has dimension $2$ and belongs to $I(\R)
  \cong 2\Z$. Therefore $\lambda\bigl(I(F) \cdot W(D,\s)\bigr) \subseteq 2\Z \oplus 2\Z \oplus
  2\Z$, and taking $\alpha$ successively  equal to $1$, $j$ and $k$, we see that 
  $\lambda\bigl(I(F) \cdot W(D,\s)\bigr) = 2\Z \oplus 2\Z \oplus
  2\Z$, i.e. 
  \[I(F) \cdot W(D,\s) = 2 W(D,\s) \cong  2\Z \oplus 2\Z \oplus 2\Z.\]
  In particular $I(F) \cdot W(D,\s)$ has index   $8$ in $W(D,\s)$ and thus cannot be equal
  to $I(D,\s)$ which has index $2$ in $W(D,\s)$ by \eqref{index2}.
\end{ex}

Since there may be several proper submodules of $W(A,\s)$ containing $I(F) \cdot
W(A,\s)$, it is interesting to see if one of them can be singled out
by a particular property. We present one such property below as an example, but other, more natural ones,  may very well exist.

In the remainder of this section we distinguish between hermitian forms $h$ over $(A,\s)$ and their classes $[h]$ in $W(A,\s)$.  
By the structure theory of $F$-algebras with involution and Morita theory, there exist $n\in\N$ and an $F$-division algebra with involution $(D,\vt)$ such that $W(A,\s)\cong W(M_n(D), \vt^t)$. Since rank modulo~$2$ is a Witt class invariant and since we will examine forms of
even rank we may also assume for convenience of notation that $(A,\s)=(M_n(D), \vt^t)$.

For $a\in \Sym(D,\vt)^\x$ we define the hermitian form
\[h_a:D^n\x D^n\too M_n(D),\ (x,y)\mapstoo \vt(x)^tay\]
(where $x$ and $y$ are considered as $1\x n$ matrices), which is of rank one over $(A,\s)$.
Let $\xi$ be the bijection between hermitian forms over $(A,\s)$ and hermitian forms over $(D,\vt)$ induced by Morita equivalence, as described in \cite{LU2}. Then $\xi(h_a)=\la a\ra_{\vt}$.

For every $a \in \Sym(D,\vt)^\times$ we define a set of hermitian forms over $(A,\s)$ by
\[\CF_a := \{ h_{a^{i_1}}  \perp \cdots \perp h_{a^{i_\ell}}    \mid \ell \in \N \textrm{ and } i_1,
\ldots, i_\ell \in \N_0\},\]
where $\perp$ denotes the usual orthogonal sum of forms.
Let $\eta_1=h_{a^{i_1}}  \perp \cdots \perp h_{a^{i_k}} $ and 
$\eta_2=h_{a^{j_1}}  \perp \cdots \perp h_{a^{j_\ell}} $ be in $\CF_a$. We define
\[\eta_1\boxtimes \eta_2 :=   \lperp_{p=1}^k \lperp_{q=1}^\ell h_{a^{i_p+j_q}}, \]
which is again an element of $\CF_a$. 

Observe that $h_a\perp h_{-a}$ and $h_1\perp h_{-a^2}$ are hyperbolic. Indeed, they are mapped by $\xi$ to the forms $\la a, -a\ra_{\vt}$ and $\la 1, -a^2\ra_{\vt}\cong \la 1, -1\ra_{\vt}$ (since $a=\vt(a)$), respectively, which are both hyperbolic.

Observe also that $I(A,\s)$ is additively generated by the classes $[h_1\perp h_a]$ for $a\in \Sym(D,\vt)^\x$ since $\xi(I(A,\s))=I(D,\vt)$.

\begin{lemma}\label{I-product}
  Let $a \in \Sym(D,\vt)^\times$ and let $\eta_1, \eta_2 \in \CF_a$ be such that 
  $[\eta_1 \boxtimes \eta_2] \in I(A,\s)$. Then $[\eta_1] \in I(A,\s)$ or $[\eta_2] \in I(A,\s)$.
\end{lemma}

\begin{proof} 
Write $\eta_1=h_{a^{i_1}}  \perp \cdots \perp h_{a^{i_k}} $ and 
$\eta_2=h_{a^{j_1}}  \perp \cdots \perp h_{a^{j_\ell}} $. Then $\rk \eta_1=k$, $\rk \eta_2=\ell$ and $\rk (\eta_1\boxtimes \eta_2)=k\ell$. Since $[\eta_1 \boxtimes \eta_2] \in I(A,\s)$ and Witt equivalence does not 
change  the parity of the rank of forms  we know that $2$ divides $k\ell$ and the result follows.
\end{proof}

\begin{prop}\label{final}
Let $(I,N)$ be a prime m-ideal of the $W(F)$-module $W(A,\s)$ such that $2 \in I$. 
Then   $N = I(A,\s)$ if and only if for every $a \in \Sym(D,\vt)^\x$ and every $\eta_1, \eta_2 \in \CF_a$ 
$[\eta_1 \boxtimes \eta_2] \in N$ implies $[\eta_1] \in N$ or $[\eta_2] \in N$.
\end{prop}

\begin{proof} The necessary direction is \eqref{I-product}. 
For the sufficient direction we will show that  $I(A,\s) \subseteq N$, and the result will follow since
$[W(A,\s):I(A,\s)] = 2$ and $N$ is a proper submodule of $W(A,\s)$. 
Let $a \in
\Sym(D,\vt)^\times$. Since $2 \in I$ we have $[2\cdot h_a] \in N$, i.e.
$[h_a] = [-h_{a}]=[h_{-a}] \mod N$. Furthermore,
$[(h_1 \perp h_a)\boxtimes (h_1\perp h_{-a})] = [h_1\perp h_{-a}\perp h_a \perp h_{-a^2}]=0 \in N$.
Therefore
$[h_1\perp h_a] \in N$ or $[h_1\perp h_{-a}]\in N$, i.e. $[h_1\perp h_a] \in N$ since $[h_a]=[h_{-a}] \mod N$. 
Since $I(A,\s)$ is additively generated by the classes
$[h_1\perp h_a]$, we conclude that $I(A,\s) \subseteq N$.
\end{proof}

Other results linking orderings on $F$ and ideals in the context of Witt groups of quadratic field extensions and quaternion division algebras  can be found in \cite{GB-H}.

\section{Canonical Identification of $H$-Signatures and Morphisms into $\Z$}

\begin{lemma}\label{epsilon}
  Let $\rho, \tau : W(A,\s) \rightarrow \Z$ be two surjective morphisms of
  abelian groups such that $\ker \tau = \ker \rho$. Then there exists $\ve  \in
  \{-1,1\}$ such that $\rho = \ve  \tau$.
\end{lemma}

\begin{proof}
  Let $N := \ker \tau$, and let $h \in W(A,\s)$ be such that $\tau(h) = 1$. Then
  $W(A,\s) = \Z\cdot h + N$, and $\rho(W(A,\s)) = \Z\cdot \rho(h)$. Since $\rho$ is
  surjective we obtain $\rho(h) = \ve  \in \{-1,1\}$. So for $h' = k\cdot h + n$
  with $k \in \Z$ and $n \in N$, we have $\tau(h') = k$ and $\rho(h') = \ve 
  k$.
\end{proof}

\begin{defi}\label{equivalent}
  Let $(\vf, \psi_1)$ and $(\vf, \psi_2)$ be $(R,S)$-morphisms of modules from
  $M$ to $N_1$ and $N_2$ repectively. We say that $(\vf, \psi_1)$ and $(\vf,
  \psi_2)$ are \df{equivalent} if there is an isomorphism of $\im \vf$-modules
  $\vartheta : \im \psi_1
  \rightarrow \im \psi_2$ such that $\psi_2 = \vartheta \circ \psi_1$.
\end{defi}

\begin{lemma}\label{fP}
  Let $(f,g)$ be a non-trivial $(W(F),\Z)$-morphism from $W(A,\s)$ to $\Z$. Then
  there exists $P \in X_F\setminus \Nil[A,\s]$ such that $(f,g)$ and $(\sign_P, \sign^H_P)$ are
  equivalent.
\end{lemma}

\begin{proof}
As seen in the proof of \eqref{kernels}, 
$f=\sign_P$ for some $P\in X_F$.
 By \eqref{m-ideal} $(\ker f, \ker g)$ is a prime m-ideal of the $W(F)$-module $W(A,\s)$. It then follows
  from \eqref{classif} that $\ker \sign^H_P
  = \ker g$. Since $g\not=0$, $P
  \not \in \Nil[A,\s]$ and there is an isomorphism of abelian groups $\rho : \im g
  \rightarrow \Z$. Let $\tau : \im \sign^H_P \rightarrow \Z$ be an isomorphism
  of abelian groups. Then $\rho \circ g$ and $\tau \circ \sign^H_P$ are
  surjective morphisms of abelian groups from $W(A,\s)$ to $\Z$ and by 
  \eqref{epsilon} there exists $\ve  \in \{-1,1\}$ such that $\rho \circ g =
  \ve  (\tau \circ \sign^H_P)$. Therefore $g = \vt \circ \sign^H_P$, where
  $\vt = \ve  (\rho^{-1} \circ \tau) : \im \sign_P^H \rightarrow \im g$ is an isomorphism
  of groups and therefore an isomorphism of $\Z$-modules. 
  So  $(f,g)$ and $(\sign_P, \sign^H_P)$ are equivalent by \eqref{equivalent}.
\end{proof}

As a consequence we obtain that $H$-signatures  can be canonically identified with
equivalence classes of $(W(F),\Z)$-morphisms from $W(A,\s)$ to $\Z$, more precisely:

\begin{prop}
  There is a bijection between the pairs $(\sign_P, \sign^H_P)$ for $P \in X_F\setminus \Nil[A,\s]$
   and the
  equivalence classes of non-trivial $(W(F),\Z)$-morphisms of modules from $W(A,\s)$ to $\Z$,
  given by
  \begin{align*}
    \{(\sign_P, \sign^H_P) \mid P \in X_F\setminus \Nil[A,\s]\} & \too  \Hom^*_{(W(F),\Z)}(W(A,\s),
      \Z)/\simm \\
     (\sign_P, \sign^H_P) & \mapstoo  (\sign_P, \sign^H_P)/\simm
    \end{align*}
\end{prop}

\begin{proof}
  By  \eqref{fP} we know that this map is surjective, and it is injective
  since assuming $(\sign_P, \sign^H_P)/\simm\ = (\sign_Q, \sign^H_Q)/\simm$ implies $\sign_P
  = \sign_Q$, and thus $P=Q$.
\end{proof}

\section*{Acknowledgement}

We would like to thank the referee for carefully reading our paper and for his
helpful suggestions.

\end{document}